\documentclass[12pt]{amsart}
\usepackage{amsfonts,amssymb,amsmath,amsthm}
\usepackage{fullpage}
\usepackage{wasysym}

\usepackage{tikz}

\newcommand{\co}{\colon}

\def\R{\mathbb{R}}
\def\Z{\mathbb{Z}}
\def\N{\mathbb{N}}

\newcommand{\ep}{\varepsilon}
\newcommand{\de}{\delta}
\newcommand{\pd}{\partial}
\newcommand{\ga}{\gamma}
\newcommand{\la}{\lambda}

\newcommand{\be}{\begin{equation}}
\newcommand{\ee}{\end{equation}}

\newcommand{\Cone}{\operatorname{Cone}}
\newcommand{\maxcol}{\operatorname{MaxCol}}
\newcommand{\dist}{\operatorname{dist}}

\newcommand{\E}{\mathcal E}
\newcommand{\B}{\mathcal B}

\newtheorem{theorem}{\bf Theorem}[section]

\newtheorem{lemma}[theorem]{Lemma}

\theoremstyle{definition}
\newtheorem{definition}[theorem]{Definition}

\theoremstyle{remark}

\numberwithin{equation}{section}

\begin{document}

\title{Examples of exponentially many collisions in a hard ball system}

\author{Dmitri Burago}                                                          
\address{Dmitri Burago: The Pennsylvania State University,                          
Department of Mathematics, University Park, PA 16802, USA}                      
\email{burago@math.psu.edu}                                                     
                                                                                
\author{Sergei Ivanov}
\address{Sergei Ivanov:
St.\ Petersburg Department of Steklov Mathematical Institute,
Russian Academy of Sciences,
Fontanka 27, St.Petersburg 191023, Russia}
\email{svivanov@pdmi.ras.ru}

\thanks{The first author was partially supported
by NSF grant DMS-1205597.
The second author was partially supported by
RFBR grants 17-01-00128 and 20-01-00070.
}



\begin{abstract}
Consider the system of $n$ identical hard balls in $\R^3$
moving freely and colliding elastically.
We show that there exist initial conditions such that
the number of collisions is exponential in $n$.
\end{abstract}

\maketitle

\section{Introduction}

Consider the system of $n$ identical hard balls
moving freely and colliding elastically.
Since long ago the problem of counting the number of collisions
that may occur between the balls has been extensively studied
for both the system of balls confined to a box and in open space.
The problem of estimating the number of collisions
goes back to Boltzmann.
Mathematically it had been proposed by Ya.~A.~Sinai,
see \cite{G81}.
It has been studied by many mathematicians.

Denote by $\maxcol(n,d)$ the maximum number of collisions that may occur
between $n$ identical balls in $\R^d$ where simultaneous collisions
are prohibited.
This number is always finite.
The fact that the number of collisions for any initial data is finite
has been shown by Vaserstein \cite{Vas79} and Galperin \cite{G81}.
The fact that $\maxcol(n,d)$ is finite has been shown by
D. Burago, Ferleger and Kononenko \cite{BFK98c}, see also~\cite{B}.
In fact,  Theorem 1.3 in \cite{BFK98c} provides a (rough) estimate
$\maxcol(n,d)\le (32 n^{2/3})^{n^2}$ for all $d$.

Many authors studying hard ball systems used the following observation.
Instead of studying the motion of balls, that is, their centers
in $\R^d$, one can put all their coordinates together as a $dn$-tuple and study the 
motion of this point in $\R^{dn}$. 
Note that some points of $\R^{dn}$ have to be removed. 
Namely, for each pair of balls there is a set of points which corresponds to configurations of balls 
where these two balls overlap. 
These sets are cylinders; in particular, they are convex.
We denote by $\B_{d,n}$ the complement of the union of these cylinders;
it is the configuration space of our system. 
It well known that the motion of the system of balls is represented
by the billiard dynamics in~$\B_{d,n}$.
Namely, $\B_{d,n}$ is a billiard table whose walls are the boundaries of the cylinders
and the usual billiard laws govern the motion exactly
corresponding to the dynamics of the $n$ balls in $\R^d$.
We forbid trajectories hitting singularities (intersections of two or more walls),
since they correspond to simultaneous collisions in the ball system.
The bounds obtained in \cite{BFK98c} do not study the system of balls directly
but rather by analyzing billiard trajectories in complements of unions of convex bodies.
Earlier Ya.~Sinai \cite{S78} has shown
that in a polyhedral cone there is a uniform upper bound
(for all trajectories) for the number of collisions with walls.

Not much was known about the lower bounds on $\maxcol(n,d)$.
It is easy to see that $\maxcol(n,1)=\frac{n(n-1)}{2}$ and it is monotone in~$d$.
If one allows different masses of balls, even in $d=1$ the situation becomes
more complicated, see e.g.~\cite{G81}.
Beyond the trivial lower bound $\maxcol(n,d)\ge\frac{n(n-1)}{2}$,
the first result we know is by Thurston and Sandri \cite{TS64} stating that
that $\maxcol(3,2)\ge 4$ (which is not obvious).
As a matter of fact, $\maxcol(3,d)=4$ for all $d\ge 2$, see \cite{MC93} and references therein.
A cubic lower bound for $\maxcol(n,2)$ is obtained in~\cite{BD18}.
This seems to be all that has been known so far.

The main result of this paper is the following theorem.

\begin{theorem}\label{t:3D}
$\maxcol(n,3) \ge 2^{\lfloor n/2\rfloor}$ for all $n\ge 3$.
\end{theorem}

Note that the lower bound in Theorem \ref{t:3D} and the upper bound from \cite{BFK98c}
have a large gap between them but at least they are both poly-exponential. Making a better
match after we went above polynomials seem not so interesting, there is little hope
to make them match exactly, and the gap between $\exp(t)$ and $\exp(2t)$ is
also huge. In fact, we prove a somewhat better lower bound which is though 
more cumbersome, see \eqref{e:final-computation}.  To make the lower- and 
upper-bounds closer to each other, one now probably should rather concentrate
on upper bounds, there obviously is some room for improvement. 

In the proof of Theorem \ref{t:3D} we construct a trajectory with 
the desired number of collisions defined on a bounded time interval.
The continuation of this trajectory may not be defined on the entire $\R$
due to a simultaneous collision.
By a small perturbation of the initial data one can obtain a trajectory
which is defined on the entire $\R$ and with at least the same number of collisions.
Indeed, such initial data form a set of full measure in the phase space. 

The collisions in our construction occur in a very small neighborhood of one singular point on the boundary
of the configuration space (billiard table) $\B_{3,n} \subset \R^{3n}$.
We find an appropriate
singular point $q$ on the boundary of $\B_{3,n}$ and consider the tangent cone to $\B_{3,n}$ at~$q$.
The point $q$ is such that the billiard system in the cone has a trajectory with the number of collisions
we need.
By applying a homothety this trajectory can be moved arbitrarily close to
the origin of the cone. 
Then it is easy  to see that there is a nearby trajectory
in $\B_{3,n}$ with the same number of collisions, see Lemma \ref{l:from-cone-to-balls}.
The point $q$ must have very special properties.

One can see that the tangent cone to $\B_{d,n}$ at any point
$q\in\pd\B_{d,n}$
is a polyhedral cone
whose faces correspond to pairs of touching balls
in the configuration represented by~$q$.
Furthermore, the angles between faces
are bounded away from~0. In our examples the number of faces equals $m=n-1$
and the angles between faces are very close to~$\frac\pi2$.
Note that, in a cone with $m$ faces where all angles are equal to~$\frac\pi2$,
every billiard trajectory experiences no more than $m$ collisions. 
Nonetheless, it turns out that an arbitrarily small change of angles
can result in a cone admitting a billiard trajectory with exponentially many collisions,
see Lemma \ref{l:almost-right-angles}.
Using this fact we first prove a model Theorem \ref{t:ndim} 
which shows that $\maxcol(n,n-1)\ge 2^{n-1}-1$.
Its proof already contains most of the principal ideas
of the main construction.

A number of open questions are left:

1. So far we were unable to prove an analog of Theorem \ref{t:3D} in dimension~2.
The reason is the lack of flexibility in constructing configurations
with prescribed angles, like the one depicted on Figure \ref{f:qhat}.

2. We do not know any interesting lower and upper bounds on the measure
of the configurations in the phase space resulting in a large number of collisions.
(For the sake of normalization, the energy and a cube to which the positions of balls
are confined to must be fixed).
The word ``large'' is vague and could mean e.g.\ some polynomial or exponential bounds.
An upper bound on the measure would be particularly interesting.

As a matter of fact, analogous problems are more interesting not in the 
whole $\R^d$
but rather in a box where the density of balls is small enough. Of course, 
then the number
of collision is counted in unit time or by averaging $N(T)/T$. This allows one
to think about dynamical characteristics like entropies (see \cite{BFK98c}, 
\cite{BFK98d}). 

3. It seems that, if the number of collisions is ``large'',
then the overwhelming number of collisions are almost tangential. This
problem had been posed in a preprint of this paper and was essentially 
answered in \cite{DS}, using a completely different set of tools. A
Physicist would call such collisions  ``inessential''
in the sense that
they result in almost zero exchange of momenta,
energy, and directions of velocities of the balls.
However, for a Dynamical System person they may look very essential, for the
analogs of Lyapunov exponents are huge. Thus theoretically such collisions 
could make a non-trivial contribution to metric entropy (which is rather unlikely) 
or to topological entropy (which is quite possible). Note that, under reasonable 
assumptions, the topological entropy is finite \cite{BFK98d}, though the proof uses 
a compactness argument in addition to Alexandrov Geometry of $k \leq 0$, and
probably no reasonable formula for the upper bound  is known or at least can be 
found in the literature. It seems that, to answer such 
questions, one needs to look at Question 2 above along with the above mentioned
estimate on the number of almost tangental collisions in \cite{DS}.

\medskip
{\bf Notation.}
Throughout the paper we denote by $\N$ the set of positive integers,
by $\R_+$ the set of nonnegative reals, and by $\R_+^m$ the set $(\R_+)^m\subset\R^m$.
The symbol $\langle\,,\rangle$ denotes the Euclidean scalar product in $\R^m$.
For a piecewise linear function $f$ defined on an interval,
we denote by $f'(t_+)$ and $f'(t_-)$ the right and left derivatives of $f$ at~$t$.

\section{Tangent cones}
\label{sec:cones}

Consider a hard ball gas system of $n$ identical balls in~$\R^d$.
Without loss of generality we set the radii of the balls to be $\frac12$.
We denote the centers of the balls by $q_1,\dots,q_n$.
Recall that we regard a collection $(q_1,\dots,q_n)\in(\R^d)^n$ as a point $q\in\R^{dn}$.
Conversely, for a point $q\in\R^{dn}$ we denote by $q_1,\dots,q_n$
its $d$-dimensional components.
Denote by $\B_{d,n}$ the configuration space of the system, that is,
$\B_{d,n}\subset\R^{dn}$ is defined by
$$
\B_{d,n} = \{q \in\R^{dn} : |q_i-q_j|\ge 1 \text{ for all $i\ne j$} \} .
$$
This set corresponds to configurations of balls with disjoint interiors.
It is the complement of the union of round cylinders
$$
 C_{ij}=\{q\in\R^{dn} : |q_i-q_j|<1\},\qquad 1\le i<j\le n .
$$
We refer to the boundaries $\pd C_{ij}$ of these cylinders as walls.
Recall that the evolution of a system of balls corresponds to the billiard dynamics in~$\B_{d,n}$.
We consider billiard trajectories defined on various intervals
with no collisions at endpoints.
Let a trajectory $\gamma$ hit a wall at a moment~$t$
and let $\nu$ be the unit normal to the wall at $\gamma(t)$.
Then the rule ``the angle of reflection
equals the angle of incidence'' takes the form
\be\label{e:angle-of-reflection}
 \ga'(t_+) = \gamma'(t_-) - 2\langle \gamma'(t_-),\nu\rangle \nu .
\ee

\begin{definition}\label{d:cone}
Let $q\in\pd \B_{d,n}$.
We denote by $\Cone(q)$ the tangent cone of $\B_{d,n}$ at $q$
defined as follows.
The point $q$ belongs to several cylinders.
They have unit outer normal vectors at $q$
referred to as {\em normals} and denoted by $\nu_1,\dots,\nu_m$.
The tangent cone $\Cone(q)$ is the set of vectors $v\in \R^{dn}$
such that $\langle v,\nu_i\rangle\ge 0$ for all $i\in\{1,\dots,m\}$.
\end{definition}

According to this definition, $\Cone(q)$ is a convex polyhedral cone
(with cone's origin at~0)
whose faces are contained in hyperplanes orthogonal to $\nu_1,\dots,\nu_m$.
If $q\in\pd C_{ij}$, $i<j$, and $\nu \in \R^{dn}$ is the normal to $C_{ij}$ at $q$,
then
\be\label{e:normal}
\nu = \frac 1{\sqrt2} (0,\dots,0, q_i-q_j,0,\dots,0,q_j-q_i,0,\dots,0),
\ee
where the nonzero entries $q_i-q_j$ and $q_j-q_i$ are at the $i$th and $j$th positions,
respectively.
To avoid case chasing below, we use the notation $C_{ij}$ for
$i>j$ as well, that is, $C_{ij}=C_{ji}$.
In the case when $i>j$ the formula for $\nu$ is similar to \eqref{e:normal}.
In both cases the $i$th $d$-dimensional component of $\nu$ equals $q_i-q_j$,
the $j$th one equals $q_j-q_i$, and all other components are zero.

The scalar products of the normals can be computed as follows.
If $q\in \pd C_{ij}\cap \pd C_{lk}$ 
and $\nu_1$ and $\nu_2$ are the normals to $C_{ij}$ and $C_{lk}$ at~$q$, then
\be\label{e:zero-angles}
 \langle \nu_1,\nu_2\rangle = 0 \quad\text{if  $\{i,j\}\cap\{l,k\}=\emptyset$}.
\ee
If $i=l$, then
\be\label{e:normal-angles}
 \langle \nu_1,\nu_2\rangle = \frac12 \langle q_j-q_i,q_k-q_i \rangle .
\ee
The first case corresponds to configurations where two disjoint pairs of balls
touch simultaneously and in the second case the $i$th ball touches the $j$th and $k$th ones.
Recall that such configurations never occur in the dynamics we study.
The cases when $i=k$, $j=l$, or $j=k$, 
reduce to \eqref{e:normal-angles} by swapping indices in $C_{ij}$ and $C_{kl}$.

The tangent cone has a nonempty interior.
Indeed, if $q\in \pd C_{ij}$ and $\nu$ is the corresponding normal
then, by \eqref{e:normal},
$$
 \langle q, \nu\rangle = \frac1{\sqrt 2} |q_i-q_j|^2= \frac1{\sqrt 2}> 0.
$$
Hence, in the notations of Definition \ref{d:cone}, the vector $q$ has positive
scalar products with the normals $\nu_1,\dots,\nu_m$ and thus belongs
to the interior of $\Cone(q)$.

\begin{lemma}\label{l:from-cone-to-balls}
Let $q\in\B_{d,n}$ and $N\in\N$ be such that
there is a billiard trajectory in $\Cone(q)$
with $N$ collisions.
Then $\maxcol(n,d)\ge N$. 
\end{lemma}

\begin{proof}
Let $W_1,\dots,W_m$ be the walls of $\B_{d,n}$ (that is, boundaries of the cylinders)
that contain~$q$ and $\nu_1,\dots,\nu_m$ their normals at~$q$.
Let
$$
\overline W_i=\{ x\in \R^{dn} : \langle x,\nu_i\rangle=0 \}, \qquad i=1,\dots,m ,
$$
be the respective walls of the cone $K:=\Cone(q)$.
Note that the walls $W_i$ do not intersect
the interior of $K$ due to the convexity of the cylinders.

Let $\ga\co (a,b)\to K$ be a billiard trajectory in the cone
with $N$ collisions at moments $a<t_1<\dots<t_N<b$
with walls $\overline W_{i_1},\dots,\overline W_{i_N}$, respectively.
For every $\la>0$, consider a rescaled set
$
 \B(\la) := \la (\B_{d,n}-q) .
$
It is bounded by the walls $W_i(\la):=\la(W_i-q)$.
We send $\la$ to infinity, fix $t_0\in(a,t_1)$ and consider a billiard
trajectory $\ga_\la$ in $\B(\la)$ with the initial conditions
$\ga_\la(t_0)=\ga(t_0)$ and $\ga_\la'(t_0)=\ga'(t_0)$.

The walls $W_i(\la)$ converge
to $\overline W_i$ as $\la\to\infty$ in $C^1$ topology on compact sets.
To avoid lengthy discussion of general submanifold convergence, we use
the following ad hoc definition in our special case.
For every $\la>0$, the rescaled wall $W_i(\la)$ is a codimension~1
smooth submanifold of~$\R^{dn}$, it contains 0, and its tangent
hyperplane at~0 is $\overline W_i$.
Hence a part of $W_i(\la)$ near 0 is a graph of a smooth function
$f_{i,\la}\co U_{i,\la} \to (\overline W_i)^\perp\simeq\R$
where $U_{i,\la}$ is a neighborhood of~0 in~$\overline W_i$,
$f_{i,\la}(0)=0$, and $df_{i,\la}(0)=0$. Since $W_i(\la)$ is
the $\la$-rescaled copy of $W_i(1)$, we can express $f_{i,\la}$
in terms of $f_{i,1}$ as follows:
$$
 U_{i,\la} = \la U_{i,1}
$$
and
$$
 f_{i,\la}(x) = \la f_{i,1}(\la^{-1}x) , \qquad x\in  U_{i,\la}.
$$
These formulae imply that for any compact set $D\subset\overline W_i$,
the domains $U_{i,\la}$ cover $D$ for all sufficiently large $\la$
and the restriction $f_{i,\la}|_D$ goes to zero in $C^1(D)$ as $\la\to\infty$.
This is what we mean by convergence of $W_i(\la)$ to $\overline W_i$.

Fix a sequence $t_0=\tau_0<\tau_1<\dots<\tau_N<b$ such that
$t_k\in(\tau_{k-1},\tau_k)$ for all $k=1,\dots,N$.
We claim that the trajectories $\ga_\la$ converge to $\ga$ in the following sense:
for every $k$ one has $\ga_\la(\tau_k)\to\ga(\tau_k)$ and
$\ga'_\la(\tau_k)\to\ga'(\tau_k)$ as $\la\to\infty$.
We prove this by induction in~$k$. The claim is trivial
for $k=0$. Assume that it holds for $k-1$ in place of $k$
and consider the first moment $t_k^0(\la)>\tau_{k-1}$
when $\ga_\la$ hits the flat wall $\overline W_{i_k}$.
The inductive hypothesis implies that
for all sufficiently large $\la$, the moment
$t_k^0(\la)$ exists,
the interval of $\ga_\la$ between $\tau_{k-1}$ and $t_k^0(\la)$
is a straight line segment (i.e., the trajectory does not hit any walls),
and $t_k^0(\la)\to t_k$ as $\la\to\infty$.
Therefore $\ga_\la(t_k^0(\la))\to\ga(t_k)$,
and the left derivative of $\ga_\la$ at $t_k^0(\la)$
converges to that of $\ga$ at~$t_k$
as $\la\to\infty$.

Recall that $W_{i_k}(\la)$ is the graph of a smooth function
$f_{i_k,\la}$ defined over a large region in~$\overline W_{i_k}$,
and the functions $f_{i_k,\la}$ tend to 0 along with their derivatives
as $\la\to\infty$.
By an elementary analysis it follows that $\ga_\la$ hits
$W_{i_k}$ at some moment $t_k(\la)>t^0_k(\la)$
such that $t_k(\la)-t^0_k(\la)\to 0$ as $\la\to\infty$.
Thus  $t_k(\la)\to t_k$,  $\ga_\la(t_k(\la))\to \ga(t_k)$
and $\ga_\la'(t_k(\la)_-)\to \ga'({t_k}_-)$ as $\la\to\infty$.
The tangent direction of $W_{i_k}(\la)$ at $\ga_\la(t_k(\la))$
converges to the direction of $\overline W_{i_k}$
since it is determined by the first derivative of $f_{i_k,\la}$.
Hence the velocity of $\ga_\la$ after the collision with
also converges: $\ga_\la'(t_k(\la)_+)\to \ga'({t_k}_+)$ as $\la\to\infty$.
If $\la$ is large enough, it follows that
$\ga_\la$ does not hit any walls on the interval $[t_k(\la),\tau_k]$,
$\ga_\la(\tau_k)\to\ga(\tau_k)$ and
$\ga'_\la(\tau_k) \to \ga'(\tau_k)$. This completes the induction step
and thus proves the claim. 

Moreover the argument implies that,
for a sufficiently large $\la$,
the trajectory $\ga_\la$ is well-defined on an interval
$(\tau_0,\tau_N)$ and experiences $N$ collisions
with walls $W_{i_1}(\la),\dots,W_{i_N}(\la)$ in this order.

Rescaling everything back, we obtain that there is a billiard trajectory
$\widetilde\ga$ in $\B_{d,n}$, namely the one defined by
$\widetilde\ga(t) = q+\la^{-1}\ga_\la(t)$
for a sufficiently large $\la$,
that experiences $N$ collisions on the interval $(t_0,t_N+\ep)$.
\end{proof}

Now we describe a simple example with exponentially many collisions
in high dimensions.
We do this mainly to facilitate understanding.
This example is not used
in the proof of the main theorem.
We begin with the following lemma.

\begin{lemma}\label{l:almost-right-angles}
For every $m\in\N$ and $\ep>0$ there exist a polyhedral cone $K\subset\R^m$
with $m$ faces and such that

1. All pairwise angles between faces of $K$ belong to $(\frac\pi2-\ep,\frac\pi2+\ep)$.

2. There exists a billiard trajectory $\ga\co\R\to K$ with $2^m-1$ collisions.
\end{lemma}

\begin{proof}
We argue by induction in $m$.
The base $m=1$ is trivial.
The induction step is from $m$ to $m+1$.
Let $K\subset\R^m$ be a cone from the induction hypothesis and
$\ga\co\R\to K$ a billiard trajectory with $N:=2^m-1$ collisions.
Let $t_1<\dots<t_N$ be the moments of these collisions.

Consider the cone $K\times\R\subset\R^{m+1}$ and
observe that for any two constants $C_0,C_1\in\R$ the path
${\overline\ga\co\R\to K\times\R}$ defined by
\be\label{e:gammabar}
  \overline\ga(t) = (\ga(t), C_1 - C_0t ) \in K\times\R
\ee
is a billiard trajectory in $K\times\R$.
We choose $C_0>0$ so large that the vector
$$
 v := - \frac {\overline\ga'(t)}{|\overline\ga'(t)|}, \qquad t>t_N,
$$
forms an angle smaller than $\ep$ with the last coordinate vector of $\R^m\times\R$.

Define a cone $\widehat K\subset\R^{m+1}=\R^m\times\R$ by 
$$
\widehat K = \{ x\in K\times\R : \langle x, v\rangle \ge 0 \} .
$$
This is a polyhedral cone with $m+1$ faces forming pairwise angles
between $\frac\pi2-\ep$ and $\frac\pi2+\ep$.
Denote by $W$ the newly added wall of this cone, 
that is, 
$$
W=\{x\in\widehat K : \langle x, v\rangle=0\} .
$$
We construct a billiard trajectory $\widehat\ga\co\R\to\widehat K$
with $2N+1=2^{m+1}-1$ collisions as follows.
Choose $C_1>0$ in \eqref{e:gammabar} so large that
$\langle\overline\ga(t_N+1),v\rangle \ge 0$.
This ensures that $\overline\ga(t)\in \widehat K$ for all $t\in(-\infty, t_N+1]$.
Then $\overline\ga$ hits $W$ at some moment $t_{N+1}\ge t_N+1$
and it hits $W$ orthogonally.
Then the path $\widehat\ga\co\R\to \widehat K$ defined by
$$
 \widehat\ga(t) =
 \begin{cases}
  \overline\ga(t) , \qquad & t\le t_{N+1}, \\
  \overline\ga(2t_{N+1}-t) , \qquad & t\ge t_{N+1},
 \end{cases}
$$
is a billiard trajectory in $\widehat K$ with $2N+1$ collisions.
This completes the induction step.
\end{proof}

\begin{theorem}\label{t:ndim}
$\maxcol(n,n-1)\ge 2^{n-1}-1$.
\end{theorem}

\begin{proof}
For $m=n-1$ and a sufficiently small $\ep>0$
construct a cone $K\subset\R^{n-1}$ as in Lemma~\ref{l:almost-right-angles}.
Let $u_1,\dots,u_{n-1}$ be the inner normals of faces of $K$.
If $\ep$ is sufficiently small then there exist
unit vectors $q_1,\dots,q_{n-1}\in\R^{n-1}$ 
such that $\langle q_i,q_j\rangle =2\langle u_i,u_j\rangle$
and $|q_i-q_j|>1$ for all $i\ne j$.
(They form a basis of $\R^{n-1}$ close to an orthonormal one).

Set $d=n-1$ and consider the configuration of balls in $\R^{n-1}$
with centers at $q_1,\dots,q_{n-1}$, and $q_n=0$.
In this configuration the $n$th ball touches all other balls
while the other ones do not touch each other.
Hence the point $q\in \B_{n-1,n}$ belongs to the walls $\pd C_{in}$, $i=1,\dots,n-1$.
Let $\nu_1,\dots,\nu_{n-1}$ be the normals to these walls at~$q$.
Then, by \eqref{e:normal-angles} and the construction of~$q$,
$$
  \langle \nu_i,\nu_j \rangle = \frac12 \langle q_i,q_j\rangle = \langle u_i,u_j\rangle,
  \qquad 1\le i<j\le n-1.
$$
Hence the frame $(\nu_1,\dots,\nu_{n-1})$ is isometric to
the frame $(u_1,\dots,u_{n-1})$.
Therefore the cone $\Cone(q)$ is isometric to $K\times\R^k$ for a suitable $k\in\N$.
Since $K$ admits a billiard trajectory with $2^{n-1}-1$ collisions,
so does $\Cone(q)$.
This and Lemma \ref{l:from-cone-to-balls} imply that there exists
a billiard trajectory in $\B_{n-1,n}$ with at least $2^{n-1}-1$ collisions.
Theorem \ref{t:ndim} follows.
\end{proof}

\section{An example in $\R^3$}

In this section we prove Theorem \ref{t:3D}. Therefore $d=3$.
We fix $n\ge 2$ for the rest of this section.
Our goal is to construct a trajectory of a system of $n$ identical balls in $\R^3$
with exponentially many collisions.
All collisions in our construction occur near a special configuration 
$\widehat q=(\widehat q_1,\dots,\widehat q_n)\in\R^{3n}$
defined as follows: we set $\widehat q_1=(0,0,0)\in\R^3$ and,
for $2\le i\le n$,
$$
 \widehat q_i =
 \begin{cases}
  (k,k-1,0) &\text{if $i=4k-2$, $k\in\Z$}, \\
  (k,k-1,-1) &\text{if $i=4k-1$, $k\in\Z$}, \\
  (k,k,0) &\text{if $i=4k$, $k\in\Z$}, \\
  (k,k,1) &\text{if $i=4k+1$, $k\in\Z$}. 
 \end{cases}
$$
This configuration is illustrated on Figure \ref{f:qhat}.
One sees that $\widehat q\in \B_{3,n}$ and $\widehat q$ has exactly $n-1$ pairs
of contacting balls.
We connect each pair of contacting balls by a segment
and denote these segments by $\widehat u_1,\dots,\widehat u_{n-1}$
as follows:
\begin{align*}
\widehat u_1 &=[\widehat q_1,\widehat q_2], \\
\widehat u_{2k}&=[\widehat q_{2k},\widehat q_{2k+1}],
\qquad k=1,2,\dots,\lfloor (n-1)/2\rfloor, \\
\widehat u_{2k+1}&=[\widehat q_{2k},\widehat q_{2k+2}],
\qquad k=1,2,\dots,\lfloor (n-2)/2\rfloor .
\end{align*}

\begin{figure}[h]
\begin{tikzpicture}[scale=0.7]
\def\hs{2.5}
\filldraw (0,0) circle (3pt) ;
\filldraw (\hs,-1) circle (3pt) ;
\filldraw (\hs,-3) circle (3pt) ;
\filldraw (2*\hs,0) circle (3pt) ;
\filldraw (2*\hs,2) circle (3pt) ;
\filldraw (3*\hs,-1) circle (3pt) ;
\filldraw (3*\hs,-3) circle (3pt) ;
\filldraw (4*\hs,0) circle (3pt) ;
\filldraw (4*\hs,2) circle (3pt) ;
\filldraw (5*\hs,-1) circle (3pt) ;
\filldraw (5*\hs,-3) circle (3pt) ;
\filldraw (7*\hs,0) circle (3pt) ;
\draw (0,0)--(\hs,-1)--(2*\hs,0)--(3*\hs,-1)--(4*\hs,0)--(5*\hs,-1)--(5.5*\hs,-0.5) ;
\draw (\hs,-1)--(\hs,-3) ;
\draw (2*\hs,0)--(2*\hs,2) ;
\draw (3*\hs,-1)--(3*\hs,-3) ;
\draw (4*\hs,0)--(4*\hs,2) ;
\draw (5*\hs,-1)--(5*\hs,-3) ;
\draw (6*\hs,-0.5) node {$\cdots$} ;
\draw (6.5*\hs,-0.5)--(7*\hs,0) ;
\draw (7*\hs,0) node[above] {$\widehat q_n$} ;
\draw (6.5*\hs,-0.5) node[below] {$\qquad\widehat u_{n-1}$} ;
\draw (0,0) node[above] {$\widehat q_1$} ;
\draw (\hs,-1) node[above] {$\widehat q_2$} ;
\draw (\hs,-3) node[below] {$\widehat q_3$} ;
\draw (2*\hs,0) node[below] {$\widehat q_4$} ;
\draw (2*\hs,2) node[above] {$\widehat q_5$} ;
\draw (3*\hs,-1) node[above] {$\widehat q_6$} ;
\draw (3*\hs,-3) node[below] {$\widehat q_7$} ;
\draw (4*\hs,0) node[below] {$\widehat q_8$} ;
\draw (4*\hs,2) node[above] {$\widehat q_9$} ;
\draw (5*\hs,-1) node[above] {$\widehat q_{10}$} ;
\draw (5*\hs,-3) node[below] {$\widehat q_{11}$} ;
\draw (0.5*\hs,-0.5) node[below] {$\widehat u_1$} ;
\draw (1*\hs,-2) node[right] {$\!\widehat u_2$} ;
\draw (1.5*\hs,-0.5) node[below] {$\widehat u_3$} ;
\draw (2*\hs,1) node[right] {$\!\widehat u_4$} ;
\draw (2.5*\hs,-0.5) node[below] {$\widehat u_5$} ;
\draw (3*\hs,-2) node[right] {$\!\widehat u_6$} ;
\draw (3.5*\hs,-0.5) node[below] {$\widehat u_7$} ;
\draw (4*\hs,1) node[right] {$\!\widehat u_8$} ;
\draw (4.5*\hs,-0.5) node[below] {$\widehat u_9$} ;
\draw (5*\hs,-2) node[right] {$\!\widehat u_{10}$} ;
\end{tikzpicture}
\caption{The configuration $\widehat q$ for $n=4k$ and its graph of ball contacts. 
This is a projection of a 3-dimensional configuration to the plane.
The points $\widehat q_1,\widehat q_2,\widehat q_4,\widehat q_6,\widehat q_8,\widehat q_{10},\dots$
lie in the $xy$-plane.
The points $\widehat q_5,\widehat q_9,\dots$ lie above $\widehat q_4,\widehat q_8,\dots$, respectively,
and $\widehat q_3,\widehat q_7,\widehat q_{11},\dots$ 
are beneath $\widehat q_2,\widehat q_6,\widehat q_{10},\dots$.
All the segments have unit lengths and meet at right angles.}
\label{f:qhat}
\end{figure}
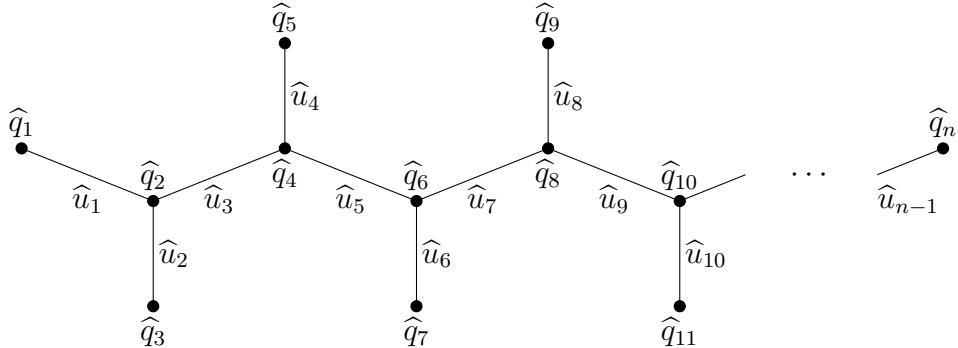

This configuration is not the one whose tangent cone admits exponentially many collisions.
Indeed, all angles between adjacent segments $\widehat u_i$ are equal to $\frac\pi2$.
Hence, by \eqref{e:zero-angles} and \eqref{e:normal-angles}, the tangent cone
$\Cone(\widehat q)$ is a right-angled cone.
This implies that a billiard trajectory in $\Cone(\widehat q)$ cannot experience
more than $n-1$ collisions.
Our plan is to construct a configuration $q\in\pd \B_{3,n}$ near $\widehat q$ whose cone
does admit trajectories with exponentially many collisions and apply
Lemma \ref{l:from-cone-to-balls} to~$q$.
(Compare with Lemma \ref{l:almost-right-angles} and Theorem \ref{t:ndim}).

We define a specific set $\E\subset\N\times\N$ by
$$
 \E = \{ (i,j) \in\N\times\N : \text{either $j=i+1$ or $i$ is odd and $j=i+2$} \} .
$$
This set is illustrated in Figure \ref{f:graph} as a set of edges of a graph
with vertices in $\N$.

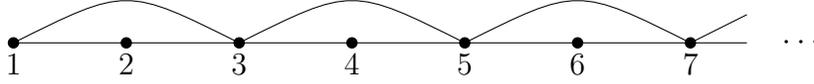
\begin{figure}[h]
\begin{tikzpicture}[scale=1.5]
\def\point#1{
\filldraw (#1-1,0) circle (1.3pt) ;
\draw (#1-1,0) node[below] {$#1$} ;
}
\point{1} \point{2} \point{3} \point{4}
\point{5} \point{6} \point{7}
\draw (0,0)--(6.5,0) ;
\draw (6,0)--(6.5,0.25) ;
\def\arc#1{
 \draw (#1-1,0) .. controls (#1,0.5) .. (#1+1,0) ;
}
\arc{1} \arc{3} \arc{5}
\draw (7,0) node {$\cdots$} ;
\end{tikzpicture}
\caption{The set $\E$.
For each $(i,j)\in \E$ the edge connecting $i$ and $j$ is depicted.
}
\label{f:graph}
\end{figure}

Let $m=n-1$.
Observe that for $1\le i<j\le m$, $(i,j)\in\E$ if and only if
the segments $\widehat u_i$ and $\widehat u_j$ meet at a common endpoint.
We denote by $\E_m$ the set of pairs $(i,j)\in\E$ such that $i,j\le m$.
We perturb our configuration by applying the following lemma.

\begin{lemma}\label{l:perturb-q}
There exists $\theta=\theta(m)>0$ such that the following holds.
For any collection of numbers $\{\alpha_{ij}\}$
indexed by pairs $(i,j)\in\E_m$
and such that $|
\alpha_{ij}-\frac\pi2|<\theta$
for all $(i,j)\in\E_m$
there exists a configuration $q\in\B_{3,m+1}$ of $m+1$ balls such that
\begin{enumerate}
\item The combinatorics of ball contacts in $q$ is the same as in $\widehat q$.
 That is, $|q_i-q_j|=1$ iff  $|\widehat q_i-\widehat q_j|=1$.
\item Let $u_1,\dots,u_m$ be the segments between the centers of pairs of touching balls of $q$
enumerated in the same way as we have enumerated $\{\widehat u_i\}$.
Then $\angle (u_i,u_j)=\alpha_{ij}$ for all $(i,j)\in\E_m$.
\end{enumerate}
\end{lemma}

\begin{proof}
This is an easy lemma.
For completeness, we provide a proof.
First consider the case when $m$ is odd.
Let $q_1=\widehat q_1$, $q_2=\widehat q_2$,
and $u_1=[q_1,q_2]$.
Then, for $i=3,5,7,\dots,m$,
let $q_{i+1}$ be the unique point in the $xy$-plane
such that $|q_{i-1}-q_{i+1}|=1$,
the segments $u_{i-2}$ and $u_i:=[q_{i-1},q_{i+1}]$
satisfy $\angle (u_{i-2},u_i)=\alpha_{i-2,i}$,
and they form a triangle oriented in the same 
way as the one formed by $\widehat u_{i-2}$ and $\widehat u_i$.

Finally, for $i=2,4,6,\dots,m-1$,
let $q_{i+1}$ be the unique point in $\R^3$ such that
$q_{i+1}$ lies in the same half-space as $\widehat q_{i+1}$ with respect to the $xy$-plane,
$|q_i-q_{i+1}|=1$,
and the segments $u_{i-1}$, $u_{i+1}$, and $u_i:=[q_i,q_{i+1}]$
satisfy $\angle(u_{i-1},u_i)=\alpha_{i-1,i}$
and $\angle(u_i,u_{i+1})=\alpha_{i,i+1}$.
This is possible whenever $\theta<\frac\pi6$,
since the three angles $\alpha_{i-1,i}$, $\alpha_{i,i+1}$, and $\alpha_{i-1,i+1}$
satisfy the triangle inequality and their sum is less than $2\pi$.

The resulting configuration $q\in\R^{3n}$ tends to $\widehat q$
as $\alpha_{ij}\to\frac\pi2$. Thus if $\theta$ is sufficiently small
then $|q_i-q_j|>1$ for all $i,j$
such that $|\widehat q_i-\widehat q_j|>1$.

In the case when $m$ is even, apply the above construction to $m+1$ in place of $m$,
assuming that $\alpha_{m,m+1}=\alpha_{m-1,m+1}=\frac\pi2$,
and then remove the point $q_{m+2}$.
\end{proof}

Let $q$ be a configuration constructed in Lemma \ref{l:perturb-q}
(for a sufficiently small $\theta$ and a collection of angles $\{\alpha_{ij}\}$
to be specified later).
Define $K=\Cone(q)$.
Each wall of $K$ corresponds to a pair of touching balls in~$q$.
We enumerate these walls in the same way as we have enumerated the segments $\{u_i\}$
and we denote by $\nu_1,\dots,\nu_m$ their respective normals.
By \eqref{e:normal-angles} and \eqref{e:zero-angles},
for $1\le i<j\le m$ we have
\be\label{e:alphas}
 \langle \nu_i,\nu_j\rangle =
 \begin{cases}
  \tfrac12 \cos\alpha_{ij} , & (i,j)\in\E \\
  0, &(i,j)\notin\E .
 \end{cases}
\ee
If $\ep$ is sufficiently small
then \eqref{e:alphas} and the assumption $|\alpha_{ij}-\frac\pi2|<\theta$
imply that the Gram matrix 
$(\langle \nu_i,\nu_j\rangle)$ is close to the identity one.
Therefore the vectors $\nu_1,\dots,\nu_m$ are linearly independent.
Hence $K$ is isometric to $K_0\times\R^{3n-m}$ where $K_0$
is the intersection of $K$ and the linear span of $\nu_1,\dots,\nu_m$.
The linear factor $\R^{3n-m}$ plays no role here and we construct a desired
billiard trajectory in $K_0$.

Note that $K_0$ is an $m$-dimensional polyhedral cone with the same
normals $\nu_1,\dots,\nu_m$ to faces.
Since the normals are linearly independent,
for every $m$-tuple $(\xi_1,\dots,\xi_m)\in \R_+^m$
there exists a unique point $x\in K_0$ such that
$\langle x,\nu_i\rangle=\xi_i$ for all~$i$.

Using this fact, we represent a billiard trajectory $\ga\co I\to K_0$, where $I$ is an interval,
by the collection of functions 
$f_i\co I\to\R_+$, $i=1,\dots,m$, given by
$f_i(t)=\langle \ga(t),\nu_i\rangle$.
In other words, $f_i(t)$ is the distance from $\ga(t)$
to the $i$th wall.
These functions are piecewise linear,
their break points (that is, discontinuity points of the derivative)
occur only at moments where one of them vanishes,
and the reflection rule \eqref{e:angle-of-reflection}
takes the following form:
If $i\in\{1,\dots,m\}$ and $t\in I$ are such that $f_i(t)=0$ then
\be\label{e:gram-rule}
 f_j'(t_+) = f_j'(t_-) - 2\langle \nu_i,\nu_j\rangle f_i'(t_-) , \qquad j=1,\dots,m .
\ee
Since $\gamma$ never hits intersections of walls,
at every moment $t\in I$
no more than one of the values $f_1(t),\dots,f_m(t)$ can vanish.

We consider a more general problem where the scalar products $\langle \nu_i,\nu_j\rangle$
in \eqref{e:gram-rule} are replaced by entries of an $m\times m$ matrix $A=(a_{ij})$
which is not assumed to be positive definite or even symmetric.

\begin{definition}
We say that an $m\times m$ matrix $A=(a_{ij})$ 
is {\em admissible} if $a_{ii}=1$ for all~$i$.
For an admissible matrix $A$, 
an {\em $A$-trajectory} is a piecewise linear function
$$
f=(f_1,\dots,f_m)\co I\to\R_+^m
$$
with finitely many break points,
where $I\subset\R$ is an interval, such that:
\begin{enumerate}
 \item[1.] No two of $f_i$'s vanish simultaneously. 
 That is, if $f_i(t)=f_j(t)=0$ for some $i$, $j$, and $t$, then $i=j$.
 \item[2.] $f$ is linear on any interval where all $f_i$'s are strictly positive. 
 \item[3.] If $i$ and $t$ are such that $f_i(t)=0$ then, for every $j\in\{1,\dots,m\}$,
\be\label{e:A-rule}
 f_j'(t_+) = f_j'(t_-) - 2a_{ij} f_i'(t_-) .
\ee
Such moments $t$ are referred to as {\em collisions}.
 \item[4.] Collisions do not occur at endpoints of $I$.
\end{enumerate}
\end{definition}

In particular, if $a_{ij}=\langle \nu_i,\nu_j\rangle$ for all $i,j$, then $A$-trajectories
correspond exactly to billiard trajectories in~$K_0$.
Due to the condition $a_{ii}=1$, the rule \eqref{e:A-rule} for $j=i$ 
takes the form $f_i'(t_+) = - f_i'(t_-)$.


We describe two ways of modifying an admissible matrix $A$
preserving the property that there is an $A$-trajectory
with many collisions.
The first one is a sufficiently small perturbation.

\begin{lemma}\label{l:perturbA}
Let $N\in\N$ and let $A$ be an admissible matrix 
such that there is an $A$-trajectory with $N$ collisions.
Then there exists $\de>0$ such that for every
admissible matrix $\widetilde A$ satisfying $\|\widetilde A-A\|<\de$
there is an $\widetilde A$-trajectory with $N$ collisions.
(Here and below the matrix norm $\|\cdot\|$ is the maximum of the absolute values
of the matrix entries).
\end{lemma}

\begin{proof}
This is yet another easy lemma.
Let $f\co (a,b)\to\R_+^n$ be an $A$-trajectory with $N$
collisions at moments $t_1<\dots<t_N$.
For $k=1,\dots,N$, let $i_k$ be the index such that $f_{i_k}(t_k)=0$.
Fix $\tau_0\in(a,t_1)$, $\tau_k\in(t_k,t_{k+1})$ for $k=1,\dots,N-1$,
and $\tau_N\in(t_N,b)$.

Clearly an $\widetilde A$-trajectory $\widetilde f$ is uniquely determined by the initial
data $(\widetilde f(\tau_0),\widetilde f'(\tau_0))$. 
For convenience we consider the matrix $\widetilde A$ as a part of the initial data.
Let $\widetilde A=(\widetilde a_{ij})$ be an admissible matrix,
$x=(x_1,\dots,x_m)\in\R_+^m$, and $v=(v_1,\dots,v_m)\in\R^m$.
If $\widetilde A$ is sufficiently close to $A$, $x$ to $f(\tau_0)$,
and $v$ to $f'(\tau_0)$, then
there exists an $\widetilde A$-trajectory 
$\widetilde f\co [\tau_0,\tau_1]\to\R_+^m$
with initial data $\widetilde f(\tau_0)=x$
and $\widetilde f'(\tau_0)=v$ and precisely one collision $f_{i_1}(\widetilde t_1)=0$
at some moment $\widetilde t_1\in(\tau_0,\tau_1)$.
Moreover the map $(\widetilde A, x, v)\mapsto (\widetilde A,f(\tau_1),f'(\tau_1))$
that sends the initial data to the terminal data
is continuous.
Indeed,  $\widetilde f$ is given by the
explicit formulae
$$
 \widetilde f_{i_1}(t) = |x_{i_1}+(t-\tau_0)v_{i_1}| 
$$
and 
$$
 \widetilde f_j(t) = x_j + (t-\tau_0) v_j - v_{i_1} \widetilde a_{i_1j} (t-\widetilde t_1 + |t-\widetilde t_1|) , \qquad j\ne i_1,
$$
where $\widetilde t_1=\tau_0-x_{i_1}/v_{i_1}$.

Applying the same argument to intervals $[\tau_{k-1},\tau_k]$, $k=1,\dots,N$,
and composing the resulting maps
one sees that, if $\widetilde A$ is sufficiently close to $A$ then
there is an $\widetilde A$-trajectory defined on $[\tau_1,\tau_N]$
with one collision on each of the intervals.
\end{proof}

The second modification of $A$ is a rescaling
described in the following lemma.

\begin{lemma}\label{l:rescaleA}
Let $A=(a_{ij})$ be an admissible matrix and 
$\la=(\la_1,\dots,\la_m)$ an $m$-tuple of positive numbers.
Define a matrix $A^\la = (a^\la_{ij})$ by
$$
  a^\la_{ij} = \frac{\la_j}{\la_i} a_{ij}, \qquad 1\le i,j\le m.
$$
Then, if $A$ admits an $A$-trajectory with $N$ collisions
then so does $A^\la$.
\end{lemma}

\begin{proof}
Note that $a^\la_{ij}=a_{ij}=1$, hence $A^\la$ is an admissible matrix.
Let $f\co I\to\R^m_+$ be an $A$-trajectory with $N$ collisions.
Define $g\co I\to\R^m_+$ by $ g_i(t) = \la_i f_i(t) $
for $i=1,\dots,m$.
Multiplying \eqref{e:A-rule} by $\la_j$ yields
$$
 g_j'(t_+) = g_j'(t_-) - 2 \frac{\la_j}{\la_i} a_{ij} g_i'(t_-)
 = g_j'(t_-) - 2 a^\la_{ij} g_i'(t_-) .
$$
Thus $g$ is an $A^\la$-trajectory.
The collisions of $g$ are at the same moments as those of~$f$.
\end{proof}

With there operations at hand, we reduce our goal to
constructing an $A$-trajectory
with many collisions
for a concrete $m\times m$ matrix $A=A_m$
whose entires $(a_{ij})$ are given by
\be\label{e:defA}
 a_{ij} =
 \begin{cases}
   1  & \text{if $i=j$} , \\
   -1 & \text{if $(i,j)\in\E_m$} , \\
   0  & \text{otherwise}.
 \end{cases}
\ee
Recall that the set $\E_m$ is not symmetric,
it includes only pairs $(i,j)$ with $i<j$.
Thus the matrix $A_m$ defined by \eqref{e:defA}
is upper-triangular.
Note that $A_m$ is a sub-matrix of $A_{m+1}$
in the sense that for $i,j\le m$, the $(i,j)$-th entries
of $A_m$ and $A_{m+1}$ are the same.

\begin{lemma}\label{l:from-A-to-balls}
Let $A_m$ be the matrix defined by \eqref{e:defA}.
Suppose that there is
an $A_m$-trajectory with $N$ collisions for some $N\in\N$.
Then $\maxcol(m+1,3)\ge N$.
\end{lemma}

\begin{proof}
We choose a finite sequence $\la=(\la_1,\dots,\la_m)$
of positive numbers that decay sufficiently fast.
The precise requirements on $\la$ are specified later.

First we require that $\la_j^2/\la_i^2 < \de$ for all $i<j$
where $\de$ is the number provided by Lemma \ref{l:perturbA}
for $A_m$ and~$N$.
Define an $m\times m$ matrix $\widetilde A=(\widetilde a_{ij})$ by
\be\label{e:atilde}
 \widetilde a_{ij} =
 \begin{cases}
   1 & \text{if $i=j$} , \\
   -1 & \text{if $(i,j)\in\E_m$} , \\
   -\la_i^2/\la_j^2  &\text{if $(j,i)\in\E_m$}, \\
   0 & \text{otherwise}.
 \end{cases}
\ee
In the third case in \eqref{e:atilde} we have $i>j$
and therefore $|\widetilde a_{ij}|<\de$.
Since the other entries of $\widetilde A$ are the same as those of $A_m$,
we have $|\widetilde a_{ij}-a_{ij}|<\de$ for all $i,j$.
Hence, by Lemma \ref{l:perturbA}, there exists an $\widetilde A$-trajectory
with at least $N$ collisions.

Now rescale $\widetilde A$ using $\la$ as in Lemma \ref{l:rescaleA}.
Denote the resulting matrix $\widetilde A^\la$ by~$B$.
The entries $(b_{ij})$ of $B$ are given by $b_{ii}=1$,
$b_{ij}=-\la_j/\la_i$ if $(i,j)\in\E_m$,
$b_{ij}=-\la_i/\la_j$ if $(j,i)\in\E_m$,
and 0 otherwise.
Hence $B$ is symmetric.

Now we require that $\la_j/\la_i<\frac12\sin\theta$
where $\theta$ is the number provided by Lemma \ref{l:perturb-q}.
For each pair $(i,j)\in\E_m$ define $\alpha_{ij}\in(\frac\pi2-\theta,\frac\pi2+\theta)$
by 
$$
 \cos \alpha_{ij} = -2\la_j/\la_i = 2 b_{ij} .
$$
Let $q\in \B_{3,m+1}$ be the configuration of balls constructed in Lemma \ref{l:perturb-q}
for this collection of angles $\{\alpha_{ij}\}$.
Let $K=\Cone(q)$ and let $\nu_1,\dots,\nu_m$ be
the normals to faces of $K$ as explained above.
Then, by \eqref{e:alphas} and the definition of $B$,
we have
$\langle\nu_i,\nu_j\rangle=b_{ij}$ for all $1\le i,j\le m$.

Therefore, as explained above, every $B$-trajectory corresponds
to a billiard trajectory in $K_0$ (and hence in $K$)
with the same number of collisions.
Thus $K$ has a billiard trajectory with at least $N$ collisions.
Finally, we apply Lemma \ref{l:from-cone-to-balls} and conclude
that $\maxcol(m+1,3)\ge N$.
\end{proof}

The rest of the paper is devoted to constructing an $A$-trajectory
with exponentially many collisions for 
the matrix $A_m$ defined by \eqref{e:defA}.
Our plan is to first construct a generalized $A_m$-trajectory
where simultaneous collision of certain type are allowed
(see Definition \ref{d:gen-A-tr}),
and then perturb the generalized $A_m$-trajectory to a obtain
a genuine one (see Lemma \ref{l:generalized-trajectory}).

\begin{definition}\label{d:gen-A-tr}
Let $A$ be an admissible $m\times m$ matrix.
A {\em generalized $A$-trajectory} is a piecewise linear
map 
$$
 f=(f_1,\dots,f_m)\co I\to \R_+^m ,
$$
where $I\subset\R$ is an interval,
such that the following holds.
\begin{enumerate}
 \item[1.]
 If $f_i(t)=f_j(t)=0$ for some $i\ne j$ and $t\in I$, then $a_{ij}=a_{ji}=0$.
 \item[2.]
For every $t\in I$ and every $j\in\{1,\dots,m\}$,
\be\label{e:A-rule-gen}
 f_j'(t_+) = f_j'(t_-) - 2\sum_{i: f_i(t)=0} a_{ij} f_i'(t_-)
\ee
where we sum over the set of all indices $i\in\{1,\dots,m\}$ such that
$f_i(t)=0$ for the given~$t$.
In particular, $f$ is linear on any interval where all $f_i$'s are positive.
 \item[3.]
If $t$ is an endpoint of $I$ then $f_i(t)>0$ for all $i$.
\end{enumerate}
By the {\em number of collisions} of a generalized $A$-trajectory $f$
we mean the total number of roots of $f_i$'s.
That is, a moment $t$ when exactly $k$ of the values $f_i(t)$
have vanished contributes $k$ to the total number of collisions.
\end{definition}

\begin{lemma}\label{l:generalized-trajectory}
Let $A$ be an admissible $m\times m$ matrix
such that there exists a generalized $A$-trajectory with $N$ collisions
(see Definition \ref{d:gen-A-tr}).
Then there exists an $A$-trajectory with $N$ collisions.
\end{lemma}

\begin{proof}
The argument is similar to that in the proof of Lemma \ref{l:perturbA}.
Let $f\co(a,b)\to\R_+^m$ be a generalized $A$-trajectory
and $t_1<\dots<t_M$ the moments of collisions.
For $k=1,\dots,M$ denote by $n_k$ the number of collisions at the moment $t_k$.
Then the total number of collisions $N$ equals $\sum n_k$.
Fix $\tau_0\in(a,t_0)$, $\tau_M\in(t_M,b)$,
and $\tau_k\in(t_k,t_{k+1})$ for $k=1,\dots,M-1$.

Just like $A$-trajectories, generalized $A$-trajectories are determined by their initial data.
We claim that for every $k\in\{1,\dots,M\}$ and any $(x,v)$ sufficiently
close to $(f(\tau_{k-1}),f'(\tau_{k-1}))$ there exists a generalized $A$-trajectory
$\widetilde f\co [\tau_{k-1},\tau_k]\to\R_+^m$
with initial data $\widetilde f(\tau_{k-1})=x$, $\widetilde f'(\tau_{k-1})=v$ and
precisely $n_k$ collisions.
Moreover the terminal data $(\widetilde f(\tau_k),\widetilde f'(\tau_k))$
depend smoothly on $(x,v)$.

To prove the claim, fix $k$ and 
define $J_k=\{i: f_i(t_k)=0 \}$.
Note that $|J_k|=n_k$.
For $(x,v)$ sufficiently close to $(f(\tau_{k-1}),f'(\tau_{k-1}))$,
define $\widetilde f\co [\tau_{k-1},\tau_k]\to\R_+^m$ by
\be\label{e:genA1}
  \widetilde f_i(t) = | x_i + (t-\tau_{k-1})v_i |, \qquad i\in J_k,
\ee
and
\be\label{e:genA2}
 \widetilde f_j(t) = x_j + (t-\tau_{k-1})v_j - \sum_{i=1}^{n_k} v_i (t-\widetilde t_{k,i} + |t-\widetilde t_{k,i}|), \qquad j\notin J_k,
\ee
where 
\be\label{e:genA3}
\widetilde t_{k,i}=\tau_{k-1}-x_i/v_i, \qquad i\in J_k,
\ee
are the roots of $f_i$'s.
Note that the roots $\widetilde t_{k,i}$ and
the terminal data $(\widetilde f(\tau_k),\widetilde f'(\tau_k))$
defined by the above formulae depend smoothly on $(x,v)$.
In particular $\widetilde t_{k,i} \in (\tau_{k-1},\tau_k)$ if the initial data $(x,v)$
is sufficiently close to $(\widetilde f(\tau_k),\widetilde f'(\tau_k))$.

The definition of a generalized $A$-trajectory implies that $a_{ij}=0$ for $i,j\in J_k$.
This ensures that $\widetilde f$ satisfies \eqref{e:A-rule-gen}
as long as $\widetilde f_j(t)>0$ for all $j\notin J_k$ and $t\in [\tau_{k-1},\tau_k]$.
The latter is true for $(x,v)=(f(\tau_{k-1}),f'(\tau_{k-1}))$
since in this case $\widetilde f=f$,
hence it is true for all $(x,v)$ sufficiently close to $(f(\tau_{k-1}),f'(\tau_{k-1}))$.
This finishes the proof of the claim.
Also observe that the roots $\widetilde t_{k,i}$
defined by \eqref{e:genA3} are distinct for almost all pairs $(x,v)$.

Similarly one shows that the initial data $(\widetilde f(\tau_{k-1}),\widetilde f'(\tau_{k-1}))$
depend smoothly on the terminal data  $(\widetilde f(\tau_k),\widetilde f'(\tau_k))$.
Thus the map that sends the initial data to the terminal data is a diffeomorphism
from a neighborhood of $(f(\tau_{k-1}),f'(\tau_{k-1}))$
to a neighborhood of $(f(\tau_k),f'(\tau_k))$.
Composing such diffeomorphisms for all $k$ we obtain that any initial data
$(x,v)$ sufficiently close to $(f(\tau_0),f'(\tau_0))$ determine
a generalized $A$-trajectory defined on $[\tau_0,\tau_M]$ with $N$ collisions.
Then by the routine of smooth topology one sees that for almost all initial
data the roots $\widetilde t_{k,i}$ are distinct for all $k$ and~$i$.

Thus a suitable perturbation of the initial data $(f(\tau_0),f'(\tau_0))$
gives us a generalized $A$-trajectory with $N$ collisions occurring at $N$ distinct moments.
Such a generalized $A$-trajectory is a genuine $A$-trajectory.
\end{proof}

It remains to construct a generalized $A_m$-trajectory,
for $A_m$ given by \eqref{e:defA},
with exponentially many collisions.
This is achieved by the following lemma.

\begin{lemma}\label{l:inductive-construction}
For $A$ defined by \eqref{e:defA},
there exists a generalized $A$-trajectory
$f\co\R\to\R_+^m$ satisfying the following conditions.
\begin{enumerate}
\item
 $|f_i'(t)|=1$ for all $i\in\{1,\dots,m\}$
 and all $t\in\R$ except the break points of $f_i$.
\item
 Denote by $T_i$ the set of all $t\in\R$ such that $f_i(t)=0$.
 Then $T_i$ is a finite arithmetic progression for every~$i$.
\item
 For all even $i=2k\le m$ one has
 $|T_i|=2^k+2^{k-1}-1$.
\item
 For all odd $i=2k+1\le m$ one has
 $|T_i|=2^{k+1}+2^k-2$.
\end{enumerate}
\end{lemma}

\begin{proof}
We argue by induction in $m$.
For the induction base $m=1$ we set $f_1(t)=|t|$.
Then $T_1=\{0\}$ and $|T_1|=1$.
We regard $T_1$ as an arithmetic progression with
common difference~1.

For the induction step, we assume that
$(f_1,\dots,f_{2k-1})$ is a generalized $A_m$-trajectory
satisfying (1)--(4) for $m=2k-1$
and prove the assertion for $m=2k$ and $m=2k+1$.
We do not change the existing $f_i$'s for $i\le 2k-1$ and just add
new functions $f_{2k}$ and $f_{2k+1}$.


By the induction hypothesis, the set $T_{2k-1}$ is a finite
arithmetic progression.
We denote its elements by $x_1<x_2<\dots<x_M$,
where $M=2^k+2^{k-1}-2$, and its common difference is denoted by~$\beta$.
We first define the set $T_{2k}$ by $T_{2k}=\{y_1,\dots,y_{M+1}\}$ where 
$$
  y_s = x_1 +(s-\tfrac32) \beta , \qquad s=1,\dots,M+1 .
$$
Note that
$$
 y_1 < x_1 < y_2 < x_2 <\dots < y_M < x_M < y_{M+1} ,
$$
$T_{2k}$ is an arithmetic progression with common difference $\beta$,
and the union $T_{2k-1}\cup T_{2k}$ is an arithmetic progression
with common difference $\beta/2$.
Now define
$$
 f_{2k}(t) = \dist(t, T_{2k} ) = \min \{ |t-y_s| : 1\le s\le M+1 \}
$$
for all $t\in\R$.
The requirements (1) and (2) for $f_{2k}$ follow from the construction.
For (3), observe that $|T_{2k}|=M+1=2^k+2^{k-1}-1$.
It remains to verify that $(f_1,\dots,f_{2k})$ is a generalized $A_{2k}$-trajectory.
Since $A_{2k}$ is upper-triangular and contains $A_{2k-1}$
as a sub-matrix, the requirements of the definition of the generalized $A$-trajectory
for the components $f_1,\dots,f_{2k-1}$ persist.
The indices $i$ such that $a_{i,2k}\ne 0$ are only $i=2k-1$ and $i=2k$.
Since $T_{2k}\cap T_{2k-1}=\emptyset$,
simultaneous collisions $f_i(t)=f_{2k}(t)=0$ can occur only if $a_{i,2k}=0$
or $i=2k$.

Let us verify \eqref{e:A-rule-gen} for $j=2k$ and all $t\in\R$. 
If $t\in T_{2k}$ then $f_{2k}'(t_-)=-1$ and $f_{2k}'(t_+)=1$.
This agrees with \eqref{e:A-rule-gen} since $a_{2k,2k}=1$.
For $t=x_s\in T_{2k-1}$, observe that $t$ is the midpoint
between $y_s$ and $y_{s+1}$,
hence it is a break point of $f_{2k}$ with
$f_{2k}'(t_-)=1$ and $f_{2k}'(t_+)=-1$.
The requirement (1) for $i=2k-1$ implies that $f_{2k-1}'(t_-)=-1$.
Since $a_{2k-1,2k}=-1$ and $a_{i,2k}=0$ for all $i<2k-1$,
these values agree with \eqref{e:A-rule-gen}.
Finally, if $t\notin T_{2k}\cup T_{2k-1}$ then 
it is not a break point of $f_{2k}$ and no $f_i$
with $a_{i,2k}\ne 0$ vanishes at~$t$.
Thus \eqref{e:A-rule-gen} is satisfied for $j=2k$ in all cases
and we have shown that $(f_1,\dots,f_{2k})$ is a generalized
$A_{2k}$-trajectory satisfying (1)--(4).

Now we construct $f_{2k+1}$.
Recall that $T_{2k-1}\cup T_{2k}$ is an arithmetic progression
of $2M+1$ elements starting at $y_1$ with common difference $\beta/2$.
We construct $f_{2k+1}$ from $T_{2k-1}\cup T_{2k}$ in the same way
as $f_{2k}$ is constructed from $T_{2k-1}$.
Namely define $T_{2k+1}=\{z_1,\dots,z_{2M+2}\}$ where
$$
 z_s = y_1+(s-\tfrac32)\tfrac\beta2
$$
and
$$
f_{2k+1}(t) = \dist(t, T_{2k+1} ) = \min \{ |t-z_s| : 1\le s\le 2M+2 \} .
$$
Note that $|T_{2k+1}|=2M+2=2^{k+1}+2^k-2$ verifying the induction step for~(4).
Using the fact that $a_{2k-1,2k+1}=a_{2k,2k+1}=-1$
and $a_{i,2k+1}=0$ for all $i<2k-1$,
we prove that $(f_1,\dots,f_{2k+1})$ is a generalized trajectory
satisfying all requirements by the same argument as in the
above proof for $(f_1,\dots,f_{2k})$.
\end{proof}

Now Lemma \ref{l:inductive-construction}, Lemma \ref{l:generalized-trajectory}, and Lemma \ref{l:from-A-to-balls}
imply that $\maxcol(m+1,3)\ge N$ where
\be\label{e:final-computation}
 N = \sum_{i=1}^m |T_i| =
  \begin{cases}
   2^{k+2}+2^{k-1}-3k-5 , &\qquad m=2k-1 , \\
   2^{k+2}+2^{k+1}-3k-6 , &\qquad m=2k .
  \end{cases}
\ee
One easily checks that $N\ge 2^k$ for all $m\ge 2$.
Since $k=\lfloor n/2 \rfloor$ for $n=m+1$, it follows that
$$
\maxcol(n,3)\ge 2^{\lfloor n/2 \rfloor}
$$
for all $n\ge 3$.

\end{document}